\documentclass[10pt]{amsart}
\usepackage{fancyhdr}
\usepackage{stmaryrd}
\usepackage{calrsfs}

\usepackage{amsmath}
\usepackage[nospace,noadjust]{cite}
\usepackage{amsfonts}
\usepackage{amssymb,enumerate}
\usepackage{amsthm}
\usepackage{cite}
\usepackage{comment}
\usepackage{color}
\usepackage[all]{xy}
\usepackage{hyperref}
\usepackage{lineno}
\usepackage{stmaryrd}

\usepackage{mathtools}
\usepackage{bm}
\usepackage{graphicx}
\usepackage{psfrag}
\usepackage{url}

\usepackage{fancyhdr}
\usepackage{tikz-cd}

\usepackage{dirtytalk}
\usepackage{float}
\usepackage{mathrsfs}

\newtheorem*{maintheorem*}{Main Theorem}
\newtheorem{theorem}{Theorem}[section]
\newtheorem{prop}[theorem]{Proposition}

\newtheorem{lemma}[theorem]{Lemma}
\newtheorem{cor}[theorem]{Corollary}

\theoremstyle{definition}
\newtheorem{definition}[theorem]{Definition}
\newtheorem{remark}[theorem]{Remark}
\newtheorem{example}[theorem]{Example}
\numberwithin{equation}{section}

\newcommand{\nn}{\mathbb{N}}
\newcommand{\qq}{\mathbb{Q}}
\newcommand{\rr}{\mathbb{R}}
\newcommand{\zz}{\mathbb{Z}}

\newcommand{\supp}{\mathsf{Supp}}

\newcommand{\uu}{\mathcal{U}}
\newcommand{\mcd}{\text{mcd}}

\newcommand{\red}{\text{red}}

\hypersetup{
	pdftitle={FD},
	colorlinks=true,
	linkcolor=blue,
	citecolor=cyan,
	urlcolor=wine
}



\begin{document}
	
	\mbox{}
	\title{Arithmetic of additively reduced monoid semidomains}

	\author{Scott T. Chapman$^1$}
	\thanks{$^1$Department of Mathematics and Statistics, Sam Houston State University}
	\address{Department of Mathematics and Statistics\\Sam Houston State University\\Huntsville, TX 77341}
	\email{scott.chapman@shsu.edu}
	
	\author{Harold Polo$^{2}$}
	\thanks{$^2$Department of Mathematics, University of Florida. Corresponding author: haroldpolo@ufl.edu}
	\address{Department of Mathematics\\University of Florida\\Gainesville, FL 32611}
	\email{haroldpolo@ufl.edu}
	
\date{\today}

\begin{abstract}
	 The term semidomain refers to a subset $S$ of an integral domain $R$, in which the pairs $(S,+)$ and $(S, \cdot)$ are semigroups with identities. If $S$ contains no additive inverses except $0$, we say that $S$ is additively reduced. By taking polynomial expressions with coefficients in $S$ and exponents in a torsion-free monoid $M$, we obtain the additively reduced monoid semidomain $S[M]$. In this paper, we investigate the factorization properties of such semidomains, providing necessary and sufficient conditions for them to be bounded factorization semidomains, finite factorization semidomains, and unique factorization semidomains. We also identify large classes of semidomains with full and infinity elasticity. Throughout the paper, we present examples to help elucidate the arithmetic of additively reduced monoid semidomains.
\end{abstract}
\medskip

\maketitle


\bigskip
\section{Introduction}
\label{sec:intro}

It is well known that the polynomial extension of a unique factorization domain (UFD) is also a unique factorization domain. In particular, the integral domain $\zz[x]$ is a UFD given that $\zz$ is a UFD. In contrast, it is not hard to find nonunit elements in the multiplicative monoid $\nn_0[x]^*$, the monoid of polynomials with positive coefficients, having multiple factorizations (see, for instance, \cite{HN1950}), even though the multiplicative monoid $\nn$ is a unique factorization monoid. In other words, the unique factorization property does not ascend from $\nn_0$ to its polynomial extension $\nn_0[x]$. This makes the factorization properties of $\nn_0[x]$ much more interesting than those of $\zz[x]$. As a matter of fact, the arithmetic of $\nn_0[x]$ and its valuations has been the subject of various articles. In~\cite{CF19}, Campanini and Facchini investigated the arithmetic and ideal structure of $\nn_0[x]$, while Brunotte~\cite{brunotte} studied the factors with positive coefficients of a given polynomial with no positive roots. In addition, Baeth and Gotti~\cite{BG20} briefly explored the factorization properties of the multiplicative monoid $\nn_0[r]^*$, where $r$ is a positive rational number.

A subset $S$ of an integral domain $R$ is called a \emph{semidomain} if the pairs $(S,+)$ and $(S, \cdot)$ form semigroups with identities. We say that $S$ is \emph{additively reduced} if $S$ contains no additive inverses except $0$. Given an additively reduced semidomain $S$ and a cancellative, commutative, and torsion-free monoid $M$, we can define the semidomain $S[M]$, which consists of polynomial expressions with coefficients in $S$ and exponents in $M$. We refer to these objects as \emph{additively reduced monoid semidomains}. The simplest additively reduced monoid semidomain is clearly $\nn_0[x]$, but other examples have also been investigated. Motivated by potential applications in control theory, Barnard et al.~\cite{BDPW1991} analyzed the relationship between the pairs of conjugate roots of a polynomial $f \in \rr_{\geq 0}[x]$ and the divisors of $f$ with positive coefficients, while Cesarz et al.~\cite{CCMS09} investigated the elasticity and delta set of $\rr_{\geq 0}[x]$. Moreover, Ponomarenko~\cite{VP15} studied the factorization properties of semigroup semirings, a class of semirings that share a substantial overlap with additively reduced semidomains.

In this paper, we aim to explore the arithmetic of additively reduced monoid semidomains, and our work is structured as follows. In Section~\ref{sec:background}, we introduce the necessary background to follow our exposition. Our first results are presented in Section~\ref{sec:atomicity}, where we provide necessary and sufficient conditions for an additively reduced monoid semidomain to be atomic and to satisfy the ACCP. Then, in Section~\ref{sec: bounded and finite fact properties}, we focus on the bounded and finite factorization properties. Specifically, we show that an additively reduced monoid semidomain $S[M]$ is a BFS (resp., an FFS) if and only if $S$ is a BFS (resp., an FFS) and $M$ is a BFM (resp., an FFM). Section~\ref{sec: factoriality} is devoted to the study of the factoriality properties of additively reduced monoid semidomains. Here we prove that an additively reduced monoid semidomain $S[M]$ is a UFS if and only if $M$ is the trivial group and $S$ is a UFS. Finally, we conclude by providing large classes of semidomains with full and infinite elasticity.

\bigskip
\section{Background}
\label{sec:background}

We now review some of the standard notation and terminology we shall be using later. Reference material on factorization theory and semiring theory can be found in the monographs \cite{GH06} and \cite{JG1999}, respectively. We let $\nn$, $\zz, \qq$, and $\rr$ denote the sets of positive integers, integers, rational numbers, and real numbers, respectively, and we set $\nn_0 := \{0\} \cup \nn$. In addition, given $r \in \rr$ and $S \subseteq \rr$, we set $S_{>r} \coloneqq \{s \in S \mid s > r\}$. We define $S_{\geq r}$ in a similar way. For $m,n \in \nn_0$, we set $\llbracket m,n \rrbracket \coloneqq \{k \in \zz \mid m \leq k \leq n\}$. Given $q \in \qq_{>0}$, we can express it uniquely as $q = d^{-1}n$, where $n, d \in \nn$ and $\gcd(n,d) = 1$. We call $n$ and $d$ the \emph{numerator} and \emph{denominator} of $q$, respectively, and denote them by $\mathsf{n}(q)$ and $\mathsf{d}(q)$, respectively.

\subsection{Monoids and Factorizations}
\smallskip

Throughout this paper, a \emph{monoid} is defined to be a semigroup with identity that is cancellative and commutative and, unless we specify otherwise, we will use multiplicative notation for monoids. For the rest of the section, let $M$ be a monoid. To denote the group of units (i.e., invertible elements) of $M$, we use the notation $M^{\times}$. When using additive notation, we will refer to this group as $\mathcal{U}(M)$. Additionally, we let $M_{\red}$ denote the quotient monoid $M/M^{\times}$. We say that $M$ is \emph{reduced} provided that the group of units of $M$ is trivial; in this case, we identify $M_{\red}$ with $M$. On the other hand, the monoid $M$ is \emph{torsion-free} if for all $b,c \in M$ and $n \in \nn$, we have that $b^n = c^n$ implies that $b = c$. Given a subset $S$ of $M$, we let $\langle S \rangle$ denote the smallest submonoid of $M$ containing $S$. We denote by $\mathcal{G}(M)$ the \emph{Grothendieck group} of $M$, which is the unique (up to isomorphism) abelian group satisfying the property that any abelian group containing a homomorphic image of $M$ must also contain a homomorphic image of $\mathcal{G}(M)$.

For elements $b,c \in M$, we say that $b$ \emph{divides} $c$ \emph{in} $M$ if there exists $b' \in M$ such that $c = b b'$; in this case, we write $b \mid_M c$, dropping the subscript whenever $M$ is the multiplicative monoid of the natural numbers. On the other hand, two elements $b,c \in M$ are \emph{associates}, which we denote by $b \simeq_M c$, provided that $b = u\cdot c$ for some $u \in M^{\times}$. A submonoid $N$ of $M$ is \emph{divisor-closed} if for every $b \in N$ and $c \in M$ the relation $c \mid_M b$ implies that $c \in N$. Let $S$ be a nonempty subset of $M$. We use the term \emph{common divisor} of $S$ to refer to an element $d\in M$ that divides all elements of $S$. On the other hand, we call a common divisor $d$ of $S$ a \emph{greatest common divisor} if it is divisible by all other common divisors of $S$. Moreover, a common divisor of $S$ is a \emph{maximal common divisor} if every greatest common divisor of $S/d$ is a unit of $M$. We denote by $\gcd_M(S)$ (resp., $\mcd_M(S)$) the set consisting of all greatest common divisors (resp., maximal common divisors) of $S$. We say that $M$ is a \emph{GCD-monoid} (resp., an \emph{MCD-monoid}) on the condition that every nonempty and finite subset of $M$ has a greatest common divisor (resp., maximal common divisor).

An element $a \in M$ is called an \emph{atom} if for every $b, c \in M$ the equality $a = bc$ implies that either $b \in M^{\times}$ or $c \in M^{\times}$. We denote by $\mathcal{A}(M)$ the set of atoms of $M$. We say that $M$ is \emph{atomic} provided that every element in $M \setminus M^{\times}$ can be written as a finite product of atoms. It is easy to verify that $M$ is atomic if and only if $M_{\red}$ is atomic. On the other hand, a subset $I$ of $M$ is an \emph{ideal} of $M$ provided that $I M \subseteq I$. An ideal $I$ of $M$ is \emph{principal} if $I = bM$ for some $b \in M$. We say that $M$ satisfies the \emph{ascending chain condition on principal ideals} (\emph{ACCP}) if every increasing sequence of principal ideals of $M$ (under inclusion) eventually terminates. It is not hard to see that if a monoid satisfies the ACCP then it is atomic.

Suppose now that the monoid $M$ is atomic. We denote by $\mathsf{Z}(M)$ the free (commutative) monoid on $\mathcal{A}(M_{\red})$ whose elements we call \emph{factorizations}. Given a factorization $z = a_1 \cdots a_\ell \in \mathsf{Z}(M)$, where $a_1, \ldots, a_\ell \in \mathcal{A}(M_{\red})$, it is said that $\ell$ is the \emph{length} of $z$. We let $\lvert z\rvert$ denote the length of a factorization $z \in \mathsf{Z}(M)$. Let $\pi \colon \mathsf{Z}(M) \to M_{\red}$ be the unique (monoid) homomorphism fixing the set $\mathcal{A}(M_{\red})$. For every element $b \in M$, the following sets associated to the element $b$ play a crucial role in the study of factorizations:
\begin{equation} \label{eq:sets of factorizations/lengths}
	\mathsf{Z}_M(b) \coloneqq \pi^{-1}(b\, \mathcal{U}(M)) \subseteq \mathsf{Z}(M) \hspace{0.1 cm}\text{ and } \hspace{0.1 cm}\mathsf{L}_M(b) \coloneqq \left\{\vert z\vert : z \in\mathsf{Z}_M(b)\right\} \subseteq \nn_0.
\end{equation}
The subscript in~\eqref{eq:sets of factorizations/lengths} is dropped when there seems to be no risk of confusion. Following~\cite{fHK92}, the monoid $M$ is called a \emph{finite factorization monoid} (\emph{FFM}) if $\mathsf{Z}(b)$ is finite for every $b \in M$, and $M$ is called a \emph{bounded factorization monoid} (\emph{BFM}) if $\mathsf{L}(b)$ is finite for all $b \in M$. It is evident that every FFM is a BFM and, by \cite[Corollary~1.3.3]{GH06}, every BFM satisfies the ACCP. Following~\cite{aZ76}, we say that~$M$ is a \emph{half-factorial monoid} (\emph{HFM}) provided that $\vert\mathsf{L}(b)\vert = 1$ for every $b \in M$. Moreover, a monoid $M$ is called \emph{factorial} or a \emph{unique factorization monoid} (\emph{UFM}) if $\vert \mathsf{Z}(b)\vert = 1$ for all $b \in M$. It is clear that a UFM is an HFM and also that an HFM is a BFM. Finally, we follow the terminology in~\cite{CCGS21} and say that~$M$ is a \emph{length-factorial monoid} (\emph{LFM}) if for every $b \in M$ and $z, z' \in \mathsf{Z}(b)$, the equality $\vert z\vert = \vert z'\vert$ implies that $z = z'$. It is obvious that if a monoid is a UFM then it is an LFM.

\subsection{Semirings and Semidomains}
\smallskip

A \emph{semiring} $S$ is a (nonempty) set endowed with two binary operations denoted by `$\cdot$' and `$+$' and called \emph{multiplication} and \emph{addition}, respectively, such that the following conditions hold:
\begin{enumerate}
	\item $(S, \cdot)$ is a commutative semigroup with an identity element denoted by $1$;
	
	\item $(S,+)$ is a monoid with its identity element denoted by $0$;
	
	\item $b \cdot (c+d)= b \cdot c + b \cdot d$ for all $b, c, d \in S$;
	
	\item $0 \cdot b = 0$ for all $b \in S$.
\end{enumerate}
We sometimes write $b c$ instead of $b \cdot c$ for elements $b,c$ in a semiring $S$. We would like to emphasize that a more general notion of a `semiring' does not usually assume the commutativity of the underlying multiplicative semigroup, but these algebraic objects are not of interest in the scope of this article. 

If $R$ and $S$ are semirings then a function $\sigma\colon R \to S$ is a \emph{semiring homomorphism} if, for all $b,c \in R$, the following conditions hold:
\begin{enumerate}
	\item $\sigma(bc) = \sigma(b)\sigma(c)$;
	
	\item $\sigma(b + c) = \sigma(b) + \sigma(c)$;
	
	\item $\sigma(1) = 1$;
	
	\item $\sigma(0) = 0$.
\end{enumerate}
We say that $\sigma$ is a \emph{semiring isomorphism} provided that $\sigma$ is injective and surjective. On the other hand, a subset $S'$ of a semiring $S$ is a \emph{subsemiring} of~$S$ if $(S',+)$ is a submonoid of $(S,+)$ that contains~$1$ and is closed under multiplication. Clearly, every subsemiring of $S$ is a semiring.

\begin{definition}
	A \emph{semidomain} is a subsemiring of an integral domain.
\end{definition}

Let $S$ be a semidomain. We say that $(S \setminus \{0\}, \cdot)$ is the \emph{multiplicative monoid} of $S$, and we denote it by $S^*$. Following standard notation from ring theory, we refer to the units of the multiplicative monoid $S^*$ simply as \emph{units}, and we refer to the units of $(S,+)$ as invertible elements without risk of ambiguity; we let $S^\times$ denote the group of units of $S$, while we let $\uu(S)$ denote the additive group of invertible elements of $S$. We denote the set of atoms of the multiplicative monoid $S^*$ as $\mathcal{A}(S)$ instead of $\mathcal{A}(S^*)$. Also, for $b,c \in S$ such that $b$ and $c$ are associates in $S^*$, we write $b \simeq_{S} c$ (instead of $b \simeq_{S^*} c$). Similarly, for $b,c \in S$ such that $b$ divides~$c$ in~$S^*$, we write $b \mid_S c$ (instead of $b \mid_{S^*} c$).

\begin{lemma} \cite[Lemma~2.2]{fghp2022} \label{lem:characterization of integral semirings}
	For a semiring $S$, the following conditions are equivalent.
	\begin{enumerate}
		\item The multiplication of $S$ extends to $\mathcal{G}(S)$ turning $\mathcal{G}(S)$ into an integral domain.
		
		\item $S$ is a semidomain.
	\end{enumerate}
\end{lemma}

Given a semidomain $S$, we let $\mathcal{F}(S)$ denote the field of fractions of $\mathcal{G}(S)$. On the other hand, we say that a semidomain $S$ is \emph{atomic} (resp., satisfies the \emph{ACCP}) if its multiplicative monoid $S^*$ is atomic (resp., satisfies the ACCP). In addition, we say that $S$ is a \emph{BFS}, \emph{FFS}, \emph{HFS}, \emph{LFS}, or \emph{UFS} provided that~$S^*$ is a BFM, FFM, HFM, LFM, or UFM, respectively. Observe that when $S$ is an integral domain, we recover the usual definitions of a UFD, a BFD, an FFD, and an HFD, which are now established notions in factorization theory.

A \emph{positive semiring}, following~\cite{BCG21}, is a subsemiring of the positive cone of $\rr$ under the standard multiplication and addition. These semirings are more tractable due to the fact that their underlying additive monoids are reduced. In~\cite{BCG21}, several examples of positive semirings are provided. It is worth noting that the class of positive semirings is contained within the class of additively reduced semidomains.

\subsection{Monoid Semirings}
\smallskip

Given a semiring $S$ and a torsion-free monoid $M$ written additively, consider the set $S[M]$ consisting of all maps $f\colon M \to S$ satisfying that the set $\{m \in M \mid f(m) \neq 0\}$ is finite. We shall conveniently represent an element $f \in S[M]$ by
\[
f = \sum_{m \in M} f(m)x^m = \sum_{i = 1}^n f(m_i)x^{m_i}\!,
\]
where the exponents $m_1, \ldots, m_n$ are the elements of $M$ whose image under $f$ is nonzero. Addition and multiplication in $S[M]$ are defined as for polynomials, and we call the elements of $S[M]$ \emph{polynomial expressions}. Under these operations, $S[M]$ is a commutative semiring, which we call the \emph{monoid semiring of $M$ over $S$} or, simply, a monoid semiring. If $S$ is a semidomain then we say that $S[M]$ is a \emph{monoid semidomain}. Observe that $S[M]$ is additively reduced provided that $S$ is additively reduced.  Since the monoid $M$ is torsion-free (and cancellative), $M$ admits a total order compatible with its monoid operation (\cite[Corollary~3.4]{rG84}). For $n \in \nn$, we say that
\[
f = s_1x^{m_1} + \cdots + s_nx^{m_n} \in S[M]^*
\] 
is written in \emph{canonical form} when $s_i \neq 0$ for every $i \in \llbracket 1,n \rrbracket$ and $m_1 > \cdots > m_n$. Observe that there is only one way to write $f$ in canonical form. As for polynomials, we call $\deg(f) \coloneqq m_1$ the \emph{degree} of $f$ and $\mathsf{c}(f) \coloneqq s_1$ the \emph{leading coefficient} of $f$. Additionally, we say that $\supp(f) \coloneqq \{m_1, \ldots, m_n\}$ is the \emph{support} of $f$, and $f$ is called a \emph{monomial} (resp., \emph{binomial}, \emph{trinomial}) if $\supp(f)$ has cardinality one (resp., two, three). 

\begin{lemma} \label{lemma: units}
	Let $S$ be a semidomain, and let $M$ be a torsion-free monoid (written additively). Then $S[M]$ is a semidomain and
	\[
	S[M]^{\times} = \left\{sx^m \,\Big\vert\, s \in S^{\times} \text{ and } \, m \in \mathcal{U}(M)\right\}.
	\]
\end{lemma}

\begin{proof}
	By virtue of \cite[Theorem~8.1]{rG84}, we have that $\mathcal{G}(S)[M]$ is an integral domain and, clearly, $S[M]$ is a subsemiring of $\mathcal{G}(S)[M]$. Hence $S[M]$ is a semidomain. The last part of our lemma is easy to verify (see, for instance, \cite[Theorem~11.1]{rG84}).
\end{proof}

\section{Atomicity and the ACCP} \label{sec:atomicity}
\smallskip

In this section, we are examining the conditions that an additively reduced monoid semidomain must satisfy to be classified as atomic and to fulfill the ACCP.

The concept of indecomposable polynomials is a crucial element in understanding the ascent of atomicity from a semidomain $S$ to its polynomial extension $S[x]$ (see \cite[Theorem~3.1]{fghp2022}). We define a polynomial in $S[x]$ to be \emph{indecomposable} if it cannot be factored as a product of two non-constant polynomials in $S[x]$. This notion has been studied in previous works \cite{fghp2022, mR93}. We now extend the notion of indecomposability to the context of monoid semidomains.

\begin{definition} \label{def: monolithic}
	Given a monoid semidomain $S[M]$, we say that a nonzero polynomial expression $f \in S[M]$ is \emph{monolithic} if $f = gh$ implies that either $g$ or $h$ is a monomial in $S[M]$. 
\end{definition}

While there are monolithic polynomials that are not indecomposable (e.g., $x^2 + x^3$ as an element of $\nn_0[x]$), indecomposable polynomials are clearly monolithic. The subsequent lemma provides insight into the significance of monolithic polynomial expressions.

\begin{lemma} \label{lemma: factors into monolithics}
	Let $S$ be an additively reduced semidomain, and let $M$ be a torsion-free monoid. Every nonzero nonunit polynomial expression in $S[M]$ factors into monolithic polynomial expressions.
\end{lemma}

\begin{proof}
	Let $f = \sum_{i = 1}^{n} s_ix^{m_i}$ be a nonzero nonunit polynomial expression in $S[M]$ written in canonical form, so $s_i \neq 0$ for every $i \in \llbracket 1,n \rrbracket$ and $m_1 > \cdots > m_n$. We proceed by induction on $n$. If $n = 1$ then $f$ is monolithic. Suppose now that all nonzero nonunit polynomial expressions whose support have cardinality strictly less than $n \in \nn_{>1}$ factor into monolithic polynomial expressions. If $f$ is not monolithic then $f = gh$, where neither $g$ nor $h$ is a monomial. Observe that since $S$ is additively reduced, we have that $\max(\vert\supp(g)\vert, \vert\supp(h)\vert) < \vert\supp(f)\vert$, from which our argument follows inductively.
\end{proof}

Now we are in a position to provide a necessary and sufficient condition for an additively reduced monoid semidomain to be atomic.

\begin{theorem} \label{prop: atomic}
	Let $S$ be an additively reduced semidomain, and let $M$ be a torsion-free monoid (written additively). Then $S[M]$ is atomic if and only if $S$ and $M$ are both atomic and
	\[
	\mcd\,(s_1, \ldots, s_n) \times \mcd\,(m_1, \ldots, m_n) \neq \emptyset
	\]
	for any monolithic polynomial expression $f = s_1x^{m_1} + \cdots + s_nx^{m_n} \in S[M]$ written in canonical form. 
\end{theorem}

\begin{proof}
	Suppose that $S[M]$ is atomic. Observe that the multiplicative monoid $N = \{sx^m \mid s \in S^* \text{ and } m \in \mathcal{U}(M)\}$ is a divisor-closed submonoid of $S[M]^*$, so it is atomic. Since $S^*_{\text{red}} \cong N_{\text{red}}$, we can conclude that $S$ is atomic. Similarly, the multiplicative monoid $H = \{sx^m \mid s \in S^{\times} \text{ and }m \in M\}$ is a divisor-closed submonoid of $S[M]^*$, so it is atomic. This implies that $M$ is also atomic as $M_{\text{red}} \cong H_{\text{red}}$. Now let $f = s_1x^{m_1} + \cdots + s_nx^{m_n} \in S[M]$ be a monolithic polynomial expression written in canonical form. Without loss of generality, assume that $f$ is not a monomial of $S[M]$ (so, in particular, $f \not\in S[M]^{\times}$). Write $f = g_1 \cdots g_t$, where $g_j \in \mathcal{A}(S[M])$ for every $j \in \llbracket 1,t \rrbracket$. Since $f$ is monolithic, there is no loss in assuming that $g_1, \ldots, g_{t - 1}$ are all monomials. Let $s = \prod_{g_i \in N} g_i$ and $y = \prod_{g_i \not\in N} g_i$, where the empty product is considered to be equal to $1$. It is easy to see that $\mathsf{c}(s) \in \mcd(s_1, \ldots, s_n)$ and that we can write $y = s'x^mg_t$ for some $s' \in S^{\times}$ and $m \in \mcd(m_1, \ldots, m_n)$. Hence $\mcd(s_1, \ldots, s_n) \times \mcd(m_1, \ldots, m_n) \neq \emptyset$. 
	
	As for the reverse implication, let us start by noticing that if $a \in \mathcal{A}(S)$ (resp., $a \in \mathcal{A}(M)$) then $a \in \mathcal{A}(S[M])$ (resp., $x^a \in \mathcal{A}(S[M])$). Now let $f = \sum_{i = 1}^{n} s_ix^{m_i} \in S[M]$ be a nonzero nonunit element written in canonical form. Since $S$ and $M$ are both atomic, there is no loss in assuming that $n > 1$. By Lemma~\ref{lemma: factors into monolithics}, we can write $f = g_1 \cdots g_k$, where $g_j$ is monolithic for each $j \in \llbracket 1,k \rrbracket$. Now fix $j \in \llbracket 1,k \rrbracket$. Thus,
	\[
	g_j = \sum_{i = 1}^{l} s_i'x^{m_i'} =  \mcd(s_1', \ldots, s_l')x^{\mcd(m_1', \ldots, m_l')}h_j
	\]
	for some $h_j \in \mathcal{A}(S[M])$. Since $S$ and $M$ are both atomic, we have that $g_j \in \langle\mathcal{A}(S[M])\rangle$ for every $j \in \llbracket 1,k \rrbracket$. Therefore, $S[M]$ is atomic.
\end{proof}

The next result follows readily from Theorem~\ref{prop: atomic}.

\begin{cor} \label{cor: strongly atomic}
	Let $S$ be an additively reduced semidomain, and let $M$ be a torsion-free monoid. Then $S[M]^*$ is an atomic MCD-monoid if and only if $S^*$ and $M$ are both atomic MCD-monoids.
\end{cor}

The validity of Theorem~\ref{prop: atomic} and Corollary~\ref{cor: strongly atomic} relies on the assumption that the semidomain $S$ is additively reduced. We now present an example (originally introduced in \cite{CG19}) where attempts to extend these results to general semidomains utterly fail.

\begin{example} \label{ex: Furstenberg + atomic is not necessarily Furstenberg}
	Fix a prime number $p$, and let $(p_n)_{n \in \nn}$ be a sequence consisting of all prime numbers different from $p$ ordered increasingly. Now set $M_p \coloneqq \langle p^{-n}p_n^{-1} \mid n \in \nn \rangle$, which is an additive submonoid of $(\qq_{\geq 0},+)$, and take $M = M_p \times M_p$. It is known that $M$ is an atomic torsion-free monoid (see \cite[page~146]{CG19}). In fact, it is easy to see that $\mathcal{A}(M_p) = \{p^{-n}p_n^{-1} \mid n \in \nn\}$, which implies that $M_p$ is atomic; consequently, $M$ is also atomic by \cite[Lemma~3.1(1)]{CG19}. An elementary argument can be used to verify that each nonzero element $m \in M_p$ has a unique representation in the form
	\begin{equation}\label{eq: unique representation}
		m = m' + \sum_{i = 1}^{n} \frac{c_i}{p^{i}p_i},
	\end{equation} 
	where $m' \in \qq_{\geq 0}$ with $\mathsf{d}(m') = p^k$ for some $k \in \nn_0$ and $0 \leq c_i \leq p_i - 1$ for each $i \in \llbracket 1,n \rrbracket$. Using representation~\eqref{eq: unique representation}, it is not hard to prove that $M_p$ is an MCD-monoid; we leave this task to the reader. Consequently, $M$ is an MCD-monoid. Now consider the monoid semidomain $F[M]$, where $F$ is a finite field of characteristic $p$. By \cite[Theorem~7.1]{rG84}, there is a ring isomorphism $F[x;\!M_p \times M_p] \cong F[y;\!M_p][z;\!M_p]$ induced by the assignment $x^{(a,b)} \mapsto y^az^b$. Consequently, we can write the elements of $F[M]$ as polynomial expressions in two variables. It is known that every nonunit factor of $f = y + z + yz$ in $F[M]$ has the form
	\[
	\left( y^{\frac{1}{p^k}} + z^{\frac{1}{p^k}} + y^{\frac{1}{p^k}}z^{\frac{1}{p^k}} \right)^t
	\]
	for some $k \in \nn_0$ and $t \in \nn$ (see \cite[page~146]{CG19}). Therefore, it is not only true that $F[M]$ is not atomic, but it is also the case that no atom in $F[M]$ divides $f$.
\end{example}

We now turn our attention to the ACCP, a property closely related to that of being atomic.

\begin{theorem} \label{prop: ACCP}
	Let $S$ be an additively reduced semidomain, and let $M$ be a torsion-free monoid (written additively). Then $S[M]$ satisfies the ACCP if and only if $S$ and $M$ satisfy both the ACCP.
\end{theorem}

\begin{proof}
	Suppose that $S[M]$ satisfies the ACCP. Since $S^*$ is a submonoid of $S[M]^*$ such that $S^{\times} = S \cap S[M]^{\times}$, we have that $S$ satisfies the ACCP. On the other hand, the multiplicative monoid $M_1 = \{x^m \mid m \in M\}$ (which is clearly isomorphic to $M$) is also a submonoid of $S[M]^*$ satisfying that $M_1^{\times} = M_1 \cap S[M]^{\times}$. Consequently, $M$ satisfies the ACCP. 
	
	Conversely, suppose that $S$ and $M$ satisfy both the ACCP, and consider the multiplicative monoid $M_2 = \{sx^m \mid s \in S^* \text{ and }m \in M\}$. Clearly, $M_2$ is a divisor-closed submonoid of $S[M]^*$. Moreover, it is easy to see that
	\[
	S^{\times} \subseteq M_2^{\times} = S[M]^{\times} = \left\{sx^m \,\Big\vert\, s \in S^{\times} \text{ and }m \in \mathcal{U}(M)\right\}\!,
	\]
	where the last equality holds by Lemma~\ref{lemma: units}. Let $(s_kx^{m_k}M_2)_{k \in \nn}$ be an ascending chain of principal ideals of $M_2$. Since the ascending chain $(s_kS)_{k \in \nn}$ of principal ideals of $S$ eventually stabilizes, there exists $n \in \nn$ such that $s_i \simeq_{S} s_n$ for every $i \geq n$. Since $S^{\times} \subseteq M_2^{\times}$, there is no loss in assuming that $s_n = s_1$ for every $n \in \nn$. Observe now that the ascending chain $(m_k + M)_{k \in \nn}$ of principal ideals of $M$ stabilizes, which implies that $(s_kx^{m_k}M_2)_{k \in \nn}$ also stabilizes. Hence $M_2$ satisfies the ACCP. By way of contradiction, assume that there exists  an ascending chain $\sigma = (f_kS[M])_{k \in \nn}$ of principal ideals of $S[M]$ such that $\sigma$ does not stabilize. If $f_t$ is a monomial for some $t \in \nn$ then $\sigma$ would stabilize because $M_2$ is a divisor-closed submonoid of $S[M]^*$ that satisfies the ACCP. As a consequence, we may assume that $f_k$ is not a monomial for any $k \in \nn$. Note that $(\vert\supp(f_k)\vert)_{k \in \nn}$ is a non-increasing sequence of natural numbers, which implies that we can also assume that $\vert\supp(f_1)\vert = \vert\supp(f_k)\vert$ for every $k \in \nn$. Hence, for every $k \in \nn$, we have $f_k = f_{k + 1}(s_{k + 1}x^{m_{k + 1}})$ for some $s_{k + 1} \in S^*$ and $m_{k + 1} \in M$. Without loss of generality, suppose that $s_{k + 1}x^{m_{k + 1}} \not\in S[M]^{\times}$ for any $k \in \nn$. Since $s_{k + 1}x^{m_{k + 1}} \not\in M_2^{\times}$ for any $k \in \nn$, we have that $\sigma^* = (\mathsf{c}(f_k)x^{\deg(f_k)}M_2)_{k \in \nn}$ is an ascending chain of ideals of $M_2$ that does not stabilize. This contradiction proves that our hypothesis is untenable. Therefore, $S[M]$ satisfies the ACCP.
\end{proof}

In Theorem~\ref{prop: ACCP}, the assumption that $S$ is additively reduced is not superfluous as $F[\mathbb{Q}]$ does not satisfy the ACCP for any field $F$ by \cite[Theorem~14.17]{rG84}. On the other hand, we can combine theorems \ref{prop: atomic} and \ref{prop: ACCP} to yield atomic semidomains that do not satisfy the ACCP as the following example illustrates. 

\begin{example} \label{ex: atomic semidomain that does not satisfy the ACCP}
	Take $r \in \qq \cap (0,1)$ with $\mathsf{n}(r) \ge 2$, and consider the additive monoid $S_r \coloneqq \langle r^n \mid n \in \nn_0 \rangle$. By \cite[Corollary~4.4]{CGG20}, the monoid $S_r$ is atomic and does not satisfy the ACCP. Moreover, it was argued in  \cite[Example~3.2]{fghp2022} that $S_r$ is an MCD-monoid. By theorems \ref{prop: atomic} and \ref{prop: ACCP}, the semidomain $\nn_0[S_r]$ is atomic and does not satisfy the ACCP.
\end{example}

\begin{remark}
	Given an additive submonoid $M$ of $\qq_{\geq 0}$, consider the additive monoid $E(M) \coloneqq \langle e^m \mid m \in M \rangle$, which is free on the set $\{e^m \mid m \in M\}$ by the Lindemann-Weierstrass Theorem stating that, for distinct algebraic numbers $\alpha_1, \ldots, \alpha_n$, the set $\{e^{\alpha_1}, \ldots, e^{\alpha_n}\}$ is linearly independent over the algebraic numbers. Observe that $E(M)$ is closed under multiplication and, consequently, it is a positive semiring. This construction has been used in the literature to construct semidomains with prescribed factorization properties (see, for instance, \cite[Example~4.15]{BG20} and \cite[Proposition~4.1]{BCG21}). Observe that the arithmetic of positive semirings of the form $E(M)$ can be better understood in the scope of the present paper since $E(M) \cong \nn[M]$ (as semirings).
\end{remark}

\section{The Bounded and Finite Factorization Properties} \label{sec: bounded and finite fact properties}
\smallskip

This section is devoted to the study of the bounded and finite factorization properties in the context of additively reduced monoid semidomains. We start with a well-known and useful characterization of BFMs.

\begin{definition}
	Given a monoid $M$, a function $\ell\colon M \to \nn_0$ is a \emph{length function} of $M$ if it satisfies the following two properties:
	\begin{enumerate}
		\item[(i)] $\ell(u) = 0$ if and only if $u \in M^{\times}$;
		
		\item[(ii)] $\ell(bc) \geq \ell(b) + \ell(c)$ for every $b,c \in M$.
	\end{enumerate}
\end{definition} 

The following result is well known.

\begin{prop} \cite[Theorem~1]{fHK92} \label{prop: length function if and only if BFM}
	A monoid $M$ is a BFM if and only if there is a length function $\ell\colon M \to \nn_0$.
\end{prop}

We are now in a position to characterize the additively reduced monoid semidomains that are BFSs.

\begin{theorem} \label{theorem: BFS}
	Let $S$ be an additively reduced semidomain, and let $M$ be a torsion-free monoid. Then $S[M]$ is a BFS if and only if $S$ is a BFS and $M$ is a BFM.
\end{theorem}

\begin{proof}
	Suppose that $S[M]$ is a BFS. As before, consider the multiplicative submonoid $N = \{sx^m \mid s \in S^* \text{ and } m \in \mathcal{U}(M)\}$ of $S[M]^*$ whose reduced monoid is isomorphic to $S^*_{\text{red}}$. By Lemma~\ref{lemma: units}, we have that $N^{\times} = N \cap S[M]^{\times}$. Consequently, the monoid $N$ is a BFM by virtue of \cite[Corollary~1.3.3]{GH06} which, in turn, implies that $S$ is a BFS. Similarly, the multiplicative submonoid $H = \{sx^m \mid s \in S^{\times} \text{ and }m \in M\}$ of $S[M]$ is a BFM as $H^{\times} = H \cap S[M]^{\times}$. Since $M_{\text{red}} \cong H_{\text{red}}$, we have that $M$ is a BFM. 
	
	Conversely, suppose that $S$ is a BFS and $M$ is a BFM. Then there exist length functions $\ell_c \colon S^* \to \nn_0$ and $\ell_e \colon M \to \nn_0$. Let us argue that the function $\ell \colon S[M]^* \to \nn_0$ given by
	\[
	\ell\left(f\right) = \ell_c\left(\mathsf{c}\left(f\right)\right) + \ell_e \left(\deg\left(f\right)\right) +\vert\,\supp\left(f\right)\vert - 1
	\]
	is also a length function. By Lemma~\ref{lemma: units}, we have that $f \in S[M]^*$ is a unit if and only if $f = sx^m$, where $s \in S^{\times}$ and $m \in \mathcal{U}(M)$. Hence $\ell(f) = 0$ if and only if $f \in S[M]^{\times}$ as the reader can easily verify. For $f, g \in S[M]^*$ we see that
	\begin{equation*}
		\begin{split}
			\ell(fg) & = \ell_c(\mathsf{c}(fg)) + \ell_e(\deg(fg)) + \left\vert\supp\left(fg\right)\right\vert - 1\\
			& \ge \ell_c(\mathsf{c}(f)) + \ell_c(\mathsf{c}(g)) + \ell_e(\deg(f)) + \ell_e(\deg(g)) + \left\vert\supp\left(f\right)\right\vert + \left\vert\supp\left(g\right)\right\vert - 2\\
			& = \ell(f) + \ell(g),
		\end{split}
	\end{equation*}
	where the inequality follows from $\ell_c$ and $\ell_e$ being both length functions along with the fact that the inequality $\vert\supp(fg)\vert \ge \vert\supp(f)\vert + \vert\supp(g)\vert - 1$ holds. Therefore, the map $\ell$ is a length function, which implies that $S[M]$ is a BFS by Proposition~\ref{prop: length function if and only if BFM}. 
\end{proof}

From the corresponding definitions, we see that an additively reduced FFS is a BFS. However, there are numerous examples in the literature illustrating that the reverse implication fails (e.g., \cite[Example~4.5]{fghp2022}). We now provide a new example of an additively reduced BFS that is not an FFS. 

\begin{example}
	The semidomain $S = \nn_0 \cup \qq_{\geq 2}$ is a BFS that is not an FFS (see \cite[Example~6.4]{BCG21}). Consider the additively reduced semidomain $R = S[x]$. By Theorem~\ref{theorem: BFS}, we have that $R$ is a BFS. However, since $R^{\times} = S^{\times}$ (Lemma~\ref{lemma: units}), $R$ is not an FFS by virtue of \cite[Theorem~1.5.6]{GH06}.
\end{example}

For the rest of the section, we focus on the finite factorization property.

\begin{definition}
	Let $g$ be an element of a torsion-free abelian group $G$ (which is additively written), and let $N_g$ be the set of positive integers $n$ such that the equation $nx = g$ has a solution in $G$. We say that $g \in G$ is \emph{of type $(0, 0, \ldots)$} provided that $N_g$ is finite. In addition, we say that $g \in G$ is \emph{of height $(0,0, \ldots)$} if $N_g$ is a singleton (i.e., $N_g = \{1\}$).
\end{definition}

\begin{theorem} \label{prop: FFS}
	Let $S$ be an additively reduced semidomain, and let $M$ be a torsion-free monoid. Then $S[M]$ is an FFS if and only if $S$ is an FFS and $M$ is an FFM. 
\end{theorem}

\begin{proof}
	Suppose that $S[M]$ is an FFS. Again, consider the multiplicative submonoid $N = \{sx^m \mid s \in S^* \text{ and } m \in \mathcal{U}(M)\}$ of $S[M]^*$. Since $N^{\times} = S[M]^{\times}$\!, the monoid $N$ is an FFM by \cite[Theorem~1.5.6]{GH06} which, in turn, implies that $S^*$ is an FFM as $S^*_{\text{red}} \cong N_{\text{red}}$. Similarly, consider the multiplicative submonoid $H = \{sx^m \mid s \in S^{\times} \text{ and } m \in M\}$ of $S[M]^*$. Since $H^{\times} = S[M]^{\times}$, the monoid $H$ is an FFM. From the fact that $M_{\text{red}} \cong H_{\text{red}}$, we conclude that $M$ is also an FFM.
	
	To tackle the reverse implication, suppose that $S^*$ and $M$ are both FFMs. For the rest of the proof, we assume that a polynomial expression in $S[M]$ is always written in canonical form. By way of contradiction, assume that $S[M]$ is not an FFS. By \cite[Proposition~1.5.5]{GH06}, there exists $f \coloneqq \sum_{i = 0}^n s_ix^{m_i} \in S[M]^*$ such that $f$ has infinitely many divisors in $S[M]$ that are pairwise non-associates. Let $g \coloneqq \sum_{j = 0}^{t} s_j'x^{m_j'}$ be an arbitrary divisor of $f$ in $S[M]$. Observe that, for every $j \in \llbracket 0,t \rrbracket$, there exists $i \in \llbracket 0,n \rrbracket$ such that $m_j' \mid_M m_i$. Moreover, the inequality $t \leq n$ holds. Since $M$ is an FFM, for some $r \in \llbracket 0,n \rrbracket$, there exists a sequence
	\[
	\sigma = \left( g^{(k)} \coloneqq \sum_{\ell = 0}^{r} s_{\ell}^{(k)}x^{m_\ell^{(k)}} \right)_{k \in \nn}
	\]
	of pairwise non-associates divisors of $f$ in $S[M]$ satisfying that $m_\ell^{(k)} \simeq_{M} m_{\ell}^{(1)}$ for every $k \in \nn$ and every $\ell \in\llbracket 0,r \rrbracket$. Since $S^*$ is also an FFM, we may assume that $s_0^{(k)} \simeq_{S} s_0^{(1)}$ for every $k \in \nn$; because the elements of $\sigma$ are pairwise non-associates, we may further assume that $s_0^{(k)} = s_0^{(1)}$ for every $k \in \nn$. Now let
	\[
	\gamma = \left( h^{(k)} \coloneqq \sum_{\ell = 0}^{t_k} c_{\ell}^{(k)}x^{o_{\ell}^{(k)}} \right)_{k \in \nn}
	\]
	such that $f = g^{(k)}h^{(k)}$ for every $k \in \nn$. Since the underlying set of the sequence $\gamma$ is an infinite subset of divisors of $f$ in $S[M]$ that are pairwise non-associates, we can assume without loss of generality that $t_k = t_1$ (from now on we refer to this quantity as $t$) and $o_\ell^{(k)} \simeq_{M} o_\ell^{(1)}$ for every $k \in \nn$ and every $\ell \in \llbracket 0,t \rrbracket$. Clearly, we have $c_0^{(k)} = c_0^{(1)}$ for each $k \in \nn$. Also, there is no loss in assuming that the equality $m_i^{(k)} + o_j^{(k)} = m_i^{(1)} + o_j^{(1)}$ holds for every $k \in \nn$, each $i \in \llbracket 0,r \rrbracket$, and each $j \in \llbracket 0,t \rrbracket$. Hence, for each $k \in \nn$, there exists $u_k \in \mathcal{U}(M)$ such that the following equalities hold:
	\begin{equation*}
		\begin{split}
			\supp\left(g^{(k)}\right) &= 	\left\{ m + u_k \,\bigg\vert\, m \in  \supp\left(g^{(1)}\right)\right\} \\  \supp\left(h^{(k)}\right) &= \left\{ m - u_k \,\bigg\vert\, m \in \supp\left(h^{(1)}\right) \right\}.
		\end{split}
	\end{equation*}
	Thus $\sigma^* = ( x^{-u_k}g^{(k)})_{k \in \nn}$ is a sequence of pairwise non-associates divisors of $f$ in $S[M]$ with the same support and leading coefficient, namely $\supp(g^{(1)})$ and $\mathsf{c}(g^{(1)})$, respectively. Let $M'$ be the submonoid of $M$ generated by the set $\supp(f) \cup \supp(g^{(1)}) \cup \supp(h^{(1)})$, and consider the monoid domain $R = F[M']$, where $F$ is a field containing $S$. Since every nonzero element of the Grothendieck group of a finitely generated torsion-free (cancellative) monoid is of type $(0,0, \ldots)$, we have that $R$ is an FFD by virtue of \cite[Proposition~3.24]{kim}. Observe that $f$, $x^{-u_k}g^{(k)}$, and $x^{u_k}h^{(k)}$ are elements of $R$ for every $k \in \nn$, which implies that there exist $i,j \in \nn$ with $i \neq j$ such that $x^{-u_i}g^{(i)} \simeq_{R} x^{-u_j}g^{(j)}$. Since $\deg(x^{-u_i}g^{(i)}) = \deg( x^{-u_j}g^{(j)})$ and $\mathsf{c}(x^{-u_i}g^{(i)}) = \mathsf{c}(x^{-u_j}g^{(j)})$, we obtain that $x^{-u_i}g^{(i)} = x^{-u_j}g^{(j)}$. This contradicts the fact that $g^{(i)}$ and $g^{(j)}$ are not associates in $S[M]$. Therefore, $S[M]$ is an FFS.
\end{proof}

A monoid $M$ is called a \emph{strong finite factorization monoid} (\emph{SFFM}) if each nonzero element of $M$ has only finitely many divisors (counting associates)\footnote{The notion of strong finite factorization was introduced by Anderson and Mullins in \cite{AM96}.}. It is easy to see that a monoid $M$ is an SFFM if and only if it is an FFM and $\vert M^{\times}\vert < \infty$. Following this definition, we say that a semidomain $S$ is a \emph{strong finite factorization semidomain} (\emph{SFFS}) provided that $S^*$ is an SFFM. In Lemma~\ref{lemma: units}, we established that, for a semidomain $S$ and a torsion-free monoid $M$ (which is written additively), the inequality $\vert S[M]^{\times}\vert < \infty$ holds if and only if $\vert S^{\times}\vert < \infty$ and $\vert\mathcal{U}(M)\vert < \infty$. We can easily deduce the following result as a direct corollary of Theorem~\ref{prop: FFS}.

\begin{cor}
	Let $S$ be an additively reduced semidomain, and let $M$ be a torsion-free monoid. Then $S[M]$ is an SFFS if and only if $S$ is an SFFS and $M$ is an SFFM. 
\end{cor}

In theorems~\ref{theorem: BFS} and \ref{prop: FFS}, the assumption that $S$ is additively reduced is not superfluous as, again, $F[\mathbb{Q}]$ does not satisfy the ACCP for any field $F$ by \cite[Theorem~14.17]{rG84}. We conclude this section providing a large class of finite factorization semidomains, but first let us introduce a definition: a positive semiring $P$ is \emph{well-ordered} if $P$ contains no decreasing sequence.

\begin{prop}
	Let $P$ be a well-ordered positive semiring. Then $P$ is an FFS.
\end{prop}

\begin{proof}
	Suppose towards a contradiction that there exists an element $b \in P^*$ such that $b$ has infinitely many non-associates (multiplicative) divisors. Then there exists an increasing sequence $(b_n)_{n \in \nn}$ consisting of non-associates divisors of $b$, which means that the underlying set of the decreasing sequence $(bb_n^{-1})_{n \in \nn}$ is a subset of $P$. This contradicts the fact that $P$ is well-ordered. By \cite[Proposition~1.5.5]{GH06}, the positive semiring $P$ is an FFS.
\end{proof}

\section{Factoriality Properties} \label{sec: factoriality}
\smallskip

An additively reduced HFS is an FFS. However, the reverse implication does not hold in general. Consider the following example.

\begin{example}
	The semidomain $S \coloneqq \nn_0[\sqrt{6}]$ is not half-factorial by \cite[Theorem~3.1]{CCMS09}. On the other hand, it is not hard to verify that if $b' + c'\sqrt{6} \,\mid_{S} b + c \sqrt{6}$, where $b + c \neq 0$, then $b' + c' \leq b + c$, which implies that $S$ is an FFS by virtue of \cite[Proposition~1.5.5]{GH06}.
\end{example}

Recall that an atomic monoid $M$ is a \emph{length-factorial monoid} (\emph{LFM}) if for all $b \in M$ and $z, z' \in \mathsf{Z}(b)$, the equality $\vert z\vert = \vert z'\vert$ implies that $z = z'$. We say that a semidomain $S$ is a \emph{length-factorial semidomain} (\emph{LFS}) if its multiplicative monoid $S^*$ is an LFM. It is evident that every UFS is an LFS. However, it is not clear whether the reverse implication holds (see \cite[Question~5.7]{fghp2022}). It is known that an integral domain is an LFS if and only if it is a UFS (\cite[Corollary~2.11]{JCWS2011}).   

In this section, we prove that an additively reduced monoid semidomain $S[M]$ is not factorial (resp., half-factorial, length-factorial), unless $M$ is the trivial group and $S$ is factorial (resp., half-factorial, length-factorial). We also provide large classes of semidomains with full and infinity elasticity. Throughout this section, we assume that a polynomial expression is always written in canonical form.

\begin{theorem} \label{prop: UFS}
	Let $S$ be an additively reduced semidomain, and let $M$ be a torsion-free monoid (written additively). The following statements hold. 
	\begin{enumerate}
		\item $S[M]$ is a UFS if and only if $S$ is a UFS and $M$ is the trivial group. 
		\item $S[M]$ is an LFS if and only if $S$ is an LFS and $M$ is the trivial group.
		\item $S[M]$ is an HFS if and only if $S$ is an HFS and $M$ is the trivial group.
	\end{enumerate}
\end{theorem}

\begin{proof}
	Let us start by proving statements $(1)$ and $(2)$. The reverse implications of both statements hold trivially. Now assume that $S[M]$ is an LFS. Suppose towards a contradiction that there exists a nonzero $m \in M$. Without loss of generality, we may assume that $m > 0$. In fact, if $m < 0$ for all $m \in M$ then we can consider the monoid semidomain $S[-M]$, where $-M = \{-m \mid m \in M\}$. Observe that $S[-M]$ is isomorphic to $S[M]$ via the isomorphism induced by the automorphism $\sigma\colon\mathcal{G}(M) \to \mathcal{G}(M)$ given by $
	\sigma(m) = -m$. Consider now the polynomial expressions 
	\[
	f_1 = x^m + 1, \hspace{.3 cm} f_2 = x^{3m} + 1, \hspace{.3 cm} f_3 = x^{2m} + x^m + 1, \hspace{.3 cm} \text{ and } \hspace{.3 cm} f_4 = x^{4m} + x^{2m} + 1
	\]
	in $S[M]$. Since $M$ is torsion-free, we have that $f_i$ and $f_j$ are not associates in $S[M]$ for $i \neq j$. Clearly, the equality $f_1f_4 = f_2f_3$ holds. Since $f_1$ and $f_2$ are binomials and $f_1(0) = f_2(0) = 1$, the polynomial expressions $f_1$ and $f_2$ are irreducibles in $S[M]$. On the other hand, observe that $\vert\mathsf{L}(f_3)\vert = 1$. Indeed, either $f_3$ is irreducible in $S[M]$ or $f_3$ is a product of two irreducible binomials. Similarly, we have $\vert\mathsf{L}(f_4)\vert = 1$. We argue that $\mathsf{L}(f_3) = \mathsf{L}(f_4)$. If we can write $f_3$ as a product of two irreducible binomials in $S[M]$ then none of the factors is an associate of $x^m + 1$ in $S[M]$ (as the reader can easily verify) and, in this case, we can also write $f_4$ as a product of two irreducible binomials in $S[M]$ using a straightforward substitution. Conversely, suppose that
	\begin{equation} \label{eq: factorization of f_4}
		\left(s_1x^{m_1} + s_2x^{m_2}\right)\left(s_3x^{m_3} + s_4x^{m_4}\right) \in \mathsf{Z}(f_4),
	\end{equation}
	where $s_1, s_2,s_3,s_4 \in S^*$ and $m_1,m_2,m_3,m_4 \in M$. From Equation~\eqref{eq: factorization of f_4}, we obtain that $s_2s_4 = 1$ and $m_2 + m_4 = 0$. Consequently, there is no loss in assuming that $s_2 = s_4 = 1$ and $m_2 = m_4 = 0$. This, in turn, implies that $m_1 = m_3 = 2m$. Hence we have $f_4 = (s_1x^{2m} + 1)(s_3x^{2m} + 1)$, which implies that $f_3 = (s_1x^{m} + 1)(s_3x^{m} + 1)$ for $s_1, s_3 \in S^*$. Observe that neither $s_1x^m + 1$ nor $s_3x^m + 1$ is an associate of $x^m + 1$ in $S[M]$. Thus $\mathsf{L}(f_3) = \mathsf{L}(f_4)$, and we can conclude that the polynomial expression
	\[
	x^{5m} + x^{4m} + x^{3m} + x^{2m} + x^{m} + 1 \in S[M]
	\]
	has two different factorizations of the same length, which contradicts that $S[M]$ is an LFS. Therefore, $M$ is the trivial group which, in turn, implies that $S$ is an LFS. If, additionally, the semidomain $S[M]$ is a UFS then $S$ is also a UFS. We can conclude that statements $(1)$ and $(2)$ hold. 
	
	To tackle the nontrivial implication of statement $(3)$, suppose that $S[M]$ is an HFS, and assume towards a contradiction that there exists a nonzero $m \in M$. Again, there is no loss in assuming that $m > 0$. Consider the polynomial expressions
	\begin{equation*}
		\begin{split}
			f_1 &= x^{4m} + x^{2m} + x^m + 1, \hspace{.8cm} f_2 =  x^{6m} + x^{5m} + x^{3m} + 1,\\
			f_3 &= x^m + 1, \hspace{.8cm} f_4 = x^{2m} + 1, \hspace{.8 cm} \text{ and }\hspace{.8cm} f_5 = x^{7m} + 2x^{4m} + 1
		\end{split}
	\end{equation*}
	in $S[M]$. Since $M$ is torsion-free, we have that $f_i$ and $f_j$ are not associates in $S[M]$ for $i \neq j$. As the reader can easily check, the equality $f_1f_2 = f_3f_4f_5$ holds. Moreover, we already established that the polynomial expressions $f_3$ and $f_4$ are irreducibles in $S[M]$. Next we argue that $f_1$ and $f_2$ are also irreducibles in $S[M]$.  
	
	\noindent \textsc{Case 1:} $f_1 = x^{4m} + x^{2m} + x^m + 1$. By way of contradiction, suppose that $f_1$ reduces in $S[M]$. Since $f_1$ is not divisible in $S[M]$ by any nonunit monomial, $f_1$ factors in $S[M]$ either as a binomial times a trinomial, or into two binomials, yielding the following two subcases. 
	
	\noindent \textsc{Case 1.1:} $f_1 = (s_1x^{m_1} + s_2x^{m_2})(s_3x^{m_3} + s_4x^{m_4} + s_5x^{m_5})$ with coefficients $s_1, s_2, s_3, s_4, s_5 \in S^*$ and exponents $m_1, m_2, m_3, m_4, m_5 \in M$. From this decomposition, we obtain the following equations:
	\begin{equation*}
		m_1 + m_3 = 4m, \hspace{.3 cm} m_2 + m_5 = 0, \hspace{.3 cm} m_2 + m_3 = 2m, \hspace{.3cm} \text{ and }\hspace{.3 cm} m_1 + m_5 = m,
	\end{equation*}
	which generate the contradiction $4m = 3m$.
	
	\noindent \textsc{Case 1.2:} $f_1 = (s_1x^{m_1} + s_2x^{m_2})(s_3x^{m_3} + s_4x^{m_4})$ with coefficients $s_1, s_2, s_3, s_4 \in S^*$ and exponents $m_1, m_2, m_3, m_4 \in M$. From this decomposition, we obtain the following equations:
	\begin{equation*}
		m_1 + m_3 = 4m, \hspace{.3 cm} m_2 + m_4 = 0, \hspace{.3cm} \text{ and }\hspace{.3 cm} m_1 + m_4 + m_2 + m_3 = 3m,
	\end{equation*}
	which is evidently a contradiction.
	
	\noindent As a consequence, we may conclude that the polynomial expression $f_1$ is irreducible in $S[M]$. To show that $f_2$ is irreducible in $S[M]$, we proceed similarly.
	
	\noindent \textsc{Case 2:} $f_2 =  x^{6m} + x^{5m} + x^{3m} + 1$. By way of contradiction, suppose that $f_2$ reduces in $S[M]$. Since $f_2$ is not divisible in $S[M]$ by any nonunit monomial, $f_2$ factors in $S[M]$ either as a binomial times a trinomial, or into two binomials, yielding the following two subcases. 
	
	\noindent \textsc{Case 2.1:} $f_2 = (s_1x^{m_1} + s_2x^{m_2})(s_3x^{m_3} + s_4x^{m_4} + s_5x^{m_5})$ with coefficients $s_1, s_2, s_3, s_4, s_5 \in S^*$ and exponents $m_1, m_2, m_3, m_4, m_5 \in M$. From this decomposition, we obtain the following equations:
	\begin{equation*}
		m_1 + m_3 = 6m, \hspace{.3 cm} m_2 + m_5 = 0, \hspace{.3 cm} m_2 + m_3 = 5m, \hspace{.3cm} \text{ and }\hspace{.3 cm} m_1 + m_5 = 3m,
	\end{equation*}
	which generate the contradiction $6m = 8m$.
	
	\noindent \textsc{Case 2.2:} $f_2 = (s_1x^{m_1} + s_2x^{m_2})(s_3x^{m_3} + s_4x^{m_4})$ with coefficients $s_1, s_2, s_3, s_4 \in S^*$ and exponents $m_1, m_2, m_3, m_4 \in M$. From this decomposition, we obtain the following equations:
	\begin{equation*}
		m_1 + m_3 = 6m, \hspace{.3 cm} m_2 + m_4 = 0, \hspace{.3cm} \text{ and }\hspace{.3 cm} m_1 + m_4 + m_2 + m_3 = 8m,
	\end{equation*}
	which is evidently a contradiction.
	
	Since the polynomial expressions $f_1$ and $f_2$ are irreducibles in $S[M]$, the element
	\[
	x^{10m} + x^{9m} + x^{8m} + 3x^{7m} + 2x^{6m} + 2x^{5m} + 2x^{4m} + x^{3m} + x^{2m} + x^{m} + 1 \in S[M]
	\]
	has a factorization of length $2$ (i.e., $f_1f_2$) and a factorization of length at least $3$ (i.e., $f_3f_4f_5$). Consequently, the semidomain $S[M]$ is not an HFS. This contradiction proves that $M$ is the trivial group which, in turn, implies that $S$ is an HFS.
\end{proof}

Based on Theorem~\ref{prop: UFS}, one might think that an additively reduced HFS is a UFS, but this is not the case.

\begin{example}
	Let $D = \zz[M]$, where $M = \langle (1,n) \mid n \in \nn \rangle \subseteq \nn_0^2$. Clearly, the (cancellative and commutative) monoid $M$ is torsion-free, which implies that $D$ is an integral domain by \cite[Theorem~8.1]{rG84}. Let $S = \{f \in \nn_0[M] \mid f(0) > 0\}$. Since $S$ is a multiplicatively subset of $D$ (i.e., a submonoid of $(D^*,\cdot)$), we can consider the localization of $D$ at $S$, which we denote by $S^{-1}D$. Set $R = (\nn_0[M] \times S)/\sim$, where $\sim$ is an equivalence relation on $\nn_0[M] \times S$ defined by $(f,g) \sim (f',g')$ if and only if $fg' = gf'$. We let $\frac{f}{g}$ denote the equivalence class of $(f,g)$. Define the following operations in $R$:
	\[
	\frac{f}{g} \cdot \frac{f'}{g'} = \frac{ff'}{gg'} \hspace{.3 cm} \text{ and } \hspace{.3 cm} \frac{f}{g} + \frac{f'}{g'} = \frac{fg' + gf'}{gg'}.
	\]
	It is routine to verify that these operations are well defined and that $(R,+,\cdot)$ is an additively reduced semiring\footnote{The localization of semirings is presented in greater generality in \cite[Chapter~11]{JG1999}.}. Let $\varphi\colon R \to S^{-1}D$ be a function given by $\varphi(f/g) = \overline{f/g}$, where $\overline{f/g}$ represents the equivalence class of $(f,g)$ as an element of $S^{-1}D$. It is easy to see that $\varphi$ is a well-defined semiring homomorphism. Since $\varphi$ is injective, $R$ is an additively reduced semidomain.
	
	It is known that $M$ is an HFM that is not a UFM (see \cite[Example~4.23]{CGG20}). Consequently, the semidomain $R$ is not a UFS. In fact, if $z = a_1 + \cdots + a_n$ and $z' = a_1' + \cdots + a_m'$ are two distinct factorizations of $m \in M$ then it is not hard to verify that
	\[
	\left(\frac{x^{a_1}}{1}\right) \cdots \left(\frac{x^{a_n}}{1}\right) \hspace{.5 cm} \text{ and } \hspace{.5 cm} \left(\frac{x^{a_1'}}{1}\right) \cdots \left(\frac{x^{a_m'}}{1}\right)
	\]
	are two distinct factorizations of the element $\frac{x^m}{1}\in R$. Next we show that $R$ is an HFS. Since all the elements of $M \setminus \{(0,0)\}$ that are not atoms are divisible by $(1,1)$ and $x^{(1,1)}/1$ is irreducible in $R$, the semidomain $R$ is atomic. Now let $f/g$ be a nonzero nonunit element of $R$. Since $f/g \simeq_{R} f/1$, there is no loss in assuming that $g = 1$. Write $f = c_kx^{m_k} + \cdots + c_1x^{m_1}$, where $m_k > \cdots > m_1 > (0,0)$ in the lexicographic order. Let
	\[
	z = \left(\frac{f_1}{g_1}\right) \cdots \left(\frac{f_n}{g_n}\right)\hspace{.5 cm} \text{ and } \hspace{.5 cm}	z' = \left(\frac{f'_1}{g'_1}\right) \cdots \left(\frac{f'_m}{g'_m}\right) 
	\]  
	be two distinct factorizations of $f/1$ in $R$, and suppose towards a contradiction that $n \neq m$. Observe that if $f'/g'$ is an atom of $R$ then writing $f' = d_lx^{o_l} + \cdots + d_1x^{o_1}$ with $o_l > \cdots > o_1 > (0,0)$ in the lexicographic order, we have that $o_1 \in \mathcal{A}(M)$ because all the elements of $M\setminus \{(0,0)\}$ that are not atoms are divisible by $(1,1)$. Consequently, the element $m_1 \in M$ has two factorizations of lengths $n$ and $m$, which contradicts that $M$ is an HFM. Therefore, $R$ is an additively reduced HFS that is not a UFS.   
\end{example}

For an atomic monoid $M$, the \emph{elasticity} of a nonunit $b \in M$, denoted by $\rho(b)$, is defined as
\[
\rho(b) = \frac{\sup \mathsf{L}(b)}{\inf \mathsf{L}(b)}.
\]
By convention, we set $\rho(u) = 1$ for every $u \in M^{\times}$\!. It is easy to see that, for all $b \in M$, we have that $\rho(b) \in \{\infty\} \cup \qq_{\ge 1}$. The \emph{elasticity} of the monoid $M$ is defined to be
\[
\rho(M) := \sup \{\rho(b) \mid b \in M \}.
\]
The \emph{set of elasticities} of $M$ is denoted by $R(M) := \{\rho(b) \mid b \in M\}$, and $M$ is said to have \emph{full elasticity} provided that, for $q \in  \qq \cap [1, \rho(M)]$, there exists an element $b \in M \setminus M^{\times}$ such that $\rho(b) = q$. Observe that a monoid $M$ is an HFM if and only if $\rho(M) = 1$ (resp., $\vert R(M)\vert = 1$). So, we can think of monoids having full and infinite elasticity as being as far as they can possibly be from being an HFM. In fact, the elasticity was first studied by Steffan~\cite{jlS86} and Valenza~\cite{rV90} with the purpose of measuring the deviation of an atomic monoid from being half-factorial. 

Next we show that an atomic monoid semidomain $S[M]$ has full and infinite elasticity provided that $(S,+)$ is reduced, $M$ is nontrivial and torsion-free, and $\mathcal{F}(S)[M]$ is a UFD, where $\mathcal{F}(S)$ denotes the quotient field of $\mathcal{G}(S)$. This generalizes \cite[Proposition~5.7]{fghp2022} in which the authors proved that a semidomain $S[x]$ has full and infinite elasticity if $S$ is additively reduced. For the rest of the section, we identify a semidomain $S$ with a subsemiring of the integral domain $\mathcal{G}(S)$ (resp., $\mathcal{F}(S)$) (see Lemma~\ref{lem:characterization of integral semirings}).

\begin{prop} \label{prop: infinite and full elasticity}
	An atomic monoid semidomain $S[M]$ has full and infinite elasticity provided that $(S,+)$ is reduced, $M$ is nontrivial and torsion-free, and $\mathcal{F}(S)[M]$ is a UFD.
\end{prop}

\begin{proof}
	Let $\mathcal{F}$ be an algebraic closure of the field $\mathcal{F}(S)$, and note that the integral domain $\mathcal{F}[M]$ is a UFD by \cite[Theorem~14.16]{rG84}. Since $(S,+)$ is reduced, $S$ contains an isomorphic copy of $\nn_0$ which, in turn, implies that $\mathcal{F}$ contains an isomorphic copy of the algebraic closure of $\qq$.  
	
	First we show that there exists a nonzero element $a \in M$ such that $x^a + b$ is irreducible in $\mathcal{F}[M]$ for every nonzero $b \in \mathcal{F}$. Since $\mathcal{F}[M]$ is a UFD, the monoid $M$ is factorial and each nonzero element of the group \,$\mathcal{U}(M)$ of invertible elements of $M$ is of type $(0,0,\ldots)$ by \cite[Theorem~14.16]{rG84}. If $M$ is reduced then it is not hard to see that, for $a \in \mathcal{A}(M)$, the polynomial expression $x^a + b$ is irreducible in $\mathcal{F}[M]$ for any nonzero $b \in \mathcal{F}$. On the other hand, if $M$ is not reduced then there is no loss in assuming that $M = \mathcal{U}(M)$. To see why our previous assumption is valid, observe that $S[\mathcal{U}(M)]^*$ is a divisor-closed submonoid of $S[M]^*$. Indeed, if there exist polynomial expressions $f = s_1x^{h_1} + \cdots + s_nx^{h_n} \in S[\mathcal{U}(M)]^*$ and $g = s_1'x^{m_1} + \cdots + s_k'x^{m_k} \in S[M]^*$ such that $g \mid_{S[M]} f$ then, since $S$ is additively reduced, for each $i \in \llbracket 1,k \rrbracket$ there exists $j \in \llbracket 1,n \rrbracket$ such that $m_i \mid_M h_j$, but $\mathcal{U}(M)$ is a divisor-closed submonoid of $M$; consequently, we have that $g \in S[\mathcal{U}(M)]^*$. Now since $S[\mathcal{U}(M)]^*$ is a divisor-closed submonoid of $S[M]^*$, the semidomain $S[\mathcal{U}(M)]$ is atomic and if $S[\mathcal{U}(M)]$ has full and infinite elasticity then $S[M]$ has full and infinite elasticity too. Consequently, we may assume that $M$ is a group satisfying that all of its nonzero elements are of type $(0,0,\ldots)$. This, in turn, implies that there exists a nonzero $a \in M$ of height $(0,0,\ldots)$. By virtue of \cite[Lemma~4.1]{matsuda}, the polynomial expression $x^a + b$ is irreducible in $\mathcal{F}[M]$ for every nonzero $b \in \mathcal{F}$.  
	
	Consider now the polynomial expression $f = x^{2a} - x^a + 1$, where $a$ is a nonzero element of $M$ such that $x^a + b$ is irreducible in $\mathcal{F}[M]$ for every nonzero $b \in \mathcal{F}$. Observe that $f$ reduces in $\mathcal{F}[M]$. In fact, $f = (x^a + \alpha)(x^a + \beta)$ for some nonzero $\alpha, \beta \in \mathcal{F}$ satisfying that $\alpha\beta = 1$ and $\alpha + \beta = -1$. We already established that $x^a + \alpha$ and $x^a + \beta$ are irreducibles (in fact, primes) in $\mathcal{F}[M]$. Note that either $n\alpha \not\in S$ for any $n \in \nn$ or $n\beta \not\in S$ for any $n \in \nn$. Indeed, if $k\alpha$ and $t\beta$ are in $S$ for some $k,t \in \nn$ then we have
	\[
	-(tk) = tk(\alpha + \beta) = t(k\alpha) + k(t\beta) \in S,
	\]     
	which contradicts that $(S,+)$ is reduced. Without loss of generality, assume that $n\alpha \not\in S$ for any $n \in \nn$. We claim that the polynomial expression $(x^a + n)^n(x^{2a} - x^a + 1)$ is irreducible in $S[M]$ for every $n \in \nn$. It follows from \cite[Lemma~2.1]{CCMS09} that, for every $n, m \in \nn$, the polynomial $(y + n)^m(y^2 - y + 1)$ is in $\nn_0[y]$ if and only if $m \geq n$. By a straightforward substitution, we obtain that, for every $n, m \in \nn$, the polynomial expression $(x^a + n)^m(x^{2a} - x^a + 1)$ is in $S[M]$ if and only if $m \geq n$. Since $S[M]$ is atomic, we can write 
	\[
	(x^a + n)^n(x^{2a} - x^a + 1) = (x^a + n)^n(x^a + \alpha)(x^a + \beta) = f_1 \cdots f_k,
	\]
	where $k \in \nn$ and $f_1, \ldots, f_k$ are irreducibles in $S[M]$. Since $n\alpha \not\in S$ for any $n \in \nn$, if $x^a + \alpha \mid_{\mathcal{F}[M]} f_j$ for some $j \in \llbracket 1,k \rrbracket$ then $x^a + \beta \mid_{\mathcal{F}[M]} f_j$. Hence if $k \geq 2$ then, for some $j \in \llbracket 1,k \rrbracket$, we have that $f_j = (x^a + n)^l(x^{2a} - x^a + 1)$ for some $0 \leq l < n$, but we already showed that this is impossible. Therefore, $(x^a + n)^n(x^{2a} - x^a + 1)$ is irreducible in $S[M]$ for every $n \in \nn$. Clearly, $x^a + 1$ and $x^{3a} + 1$ are irreducibles in $S[M]$. For $k \in \nn$ and $n \in \nn_{>1}$, consider the polynomial expression
	\[
	g = (x^a + n)^n(x^{2a} - x^a + 1)(x^a + 1)^k \in S[M].
	\]
	Since $\mathcal{F}[M]$ is a UFD and $x^a + b$ is irreducible in $\mathcal{F}[M]$ for every nonzero $b \in \mathcal{F}$, the only two factorizations of $g$ in $S[M]$ are
	\[
	[(x^a + n)^n(x^{2a} - x^a + 1)]\cdot [x^a + 1]^k \hspace{.3 cm}\text{ and } \hspace{.3 cm}[x^a + n]^n \cdot [(x^{2a} - x^a + 1)(x^a + 1)]\cdot[x^a + 1]^{k - 1}
	\]
	with lengths $k + 1$ and $k + n$, respectively. Since
	\[
	\left\{\frac{k + n}{k + 1} \;\bigg\vert\; k \in \nn \text{ and } n \in \nn_{>1}\right\} = \qq_{> 1},
	\]
	we conclude that $S[M]$ has full and infinite elasticity.
\end{proof}

We now provide two large classes of additively reduced monoid semidomains with full and infinite elasticity.

\begin{cor} \label{prop: infinite and full elasticity final}
	Let $S[M]$ be an atomic monoid semidomain such that $S$ is additively reduced and $M$ is nontrivial and torsion-free. The following statements hold.
	\begin{enumerate}
		\item If $M$ is a reduced UFM then $S[M]$ has full and infinite elasticity.
		\item If $M$ is not reduced and every nonzero element of $\mathcal{U}(M)$ is of type $(0,0,\ldots)$ then $S[M]$ has full and infinite elasticity.
	\end{enumerate}
\end{cor}

\begin{proof}
	Observe that if $M$ is a reduced UFM then $\mathcal{F}(S)[M]$ is a UFD by \cite[Theorem~14.16]{rG84}. Consequently, the statement $(1)$ follows from Proposition~\ref{prop: infinite and full elasticity}. Now suppose that $M$ is not reduced and that every nonzero element of $\mathcal{U}(M)$ is of type $(0,0,\ldots)$. We already established that $S[\mathcal{U}(M)]^*$ is a divisor-closed submonoid of $S[M]^*$, which implies that $S[\mathcal{U}(M)]$ is atomic. By \cite[Theorem~14.15]{rG84}, the integral domain $\mathcal{F}(S)[\mathcal{U}(M)]$ is a UFD. Then $S[\mathcal{U}(M)]$ has full and infinite elasticity by virtue of Proposition~\ref{prop: infinite and full elasticity}, which concludes our argument.   
\end{proof}

\begin{cor}
	For $n \in \nn$, an atomic polynomial semidomain $S[x_1, \ldots, x_n]$ (resp., $S[x_1^{\pm 1}, \ldots, x_n^{\pm 1}]$) has full and infinite elasticity provided that $(S,+)$ is reduced. 
\end{cor}

\begin{remark} \label{remark: about transfer Krull monoids}
	Transfer Krull monoids (see \cite[Section~4]{G2016} for the definition) were introduced by Geroldinger~\cite{G2016} and, since then, they have been examined across diverse settings~\cite{BR2022,GZ19,GLTZ21}. In \cite[Theorem~3.1]{GZ19}, Geroldinger and Zhong proved that transfer Krull monoids have full elasticity, which begs the question of whether the monoid semidomains described in Proposition~\ref{prop: infinite and full elasticity} are transfer Krull. To address this question, we follow the argument given by Campanini and Facchini in \cite[Remark~5.4]{CF19} to show that $\nn_0[x]$ is not transfer Krull. Let $M$ be a transfer Krull monoid, and let $S[M]$ be a monoid semidomain satisfying that $(S,+)$ is reduced, $M$ is nontrivial and torsion-free, and $\mathcal{F}(S)[M]$ is a UFD. Combining \cite[Proposition~3.2.3]{GH06} and \cite[Lemma~6.4.4]{GH06}, it is possible to show that, for a fixed atom $a \in \mathcal{A}(M)$, the set of lengths $\{\mathsf{L}(aa') \mid a' \in \mathcal{A}(M)\}$ is bounded by a constant depending only on the atom $a$. However, we showed (as part of the proof of Proposition~\ref{prop: infinite and full elasticity}) the following two facts about $S[M]$:
	\begin{enumerate}
		\item $(x^a + n)^n(x^{2a} - x^a + 1)$ is irreducible in $S[M]$ for every $n \in \nn$;
		\item the polynomial expression $(x^a + n)^n(x^{2a} - x^a + 1)(x^a + 1) \in S[M]$ (with $n \in \nn_{>1}$) has exactly two factorizations of lengths $2$ and $n + 2$.
	\end{enumerate}
	Consequently, the set $\left\{\mathsf{L}((x^a + 1)a') \mid a' \in \mathcal{A}(S[M])\right\}$ is not bounded, which implies that $S[M]$ is not transfer Krull.
\end{remark}

We conclude this section by showing that the reverse implication of Proposition~\ref{prop: infinite and full elasticity} does not hold as the following example illustrates.

\begin{example}
	Let $M = \langle (3/2)^n \mid n \in \nn_0 \rangle \subseteq (\qq_{\geq 0}, +)$, and consider the monoid semidomain $\nn_0[M]$. Since $M$ is an FFM (\cite[Theorem~5.6]{fG19}), the semidomain $\nn_0[M]$ is atomic (in fact, an FFS) by Theorem~\ref{prop: FFS}. It was proved in \cite[Proposition~4.4(3)]{ScGG2019} that $M$ has full and infinite elasticity, which implies that $\nn_0[M]$ has full and infinite elasticity too. However, since $M$ is not a UFM, the domain $F[M]$ is not a UFD for any field $F$ containing $\nn_0$ (\cite[Theorem~14.7]{rG84}).
\end{example}

\section{Acknowledgments}
\smallskip

The authors express their gratitude to an anonymous referee for their meticulous review of the manuscript and valuable feedback that significantly enhanced the quality of this paper. Additionally, the authors would like to extend their appreciation to Alfred Geroldinger and the referee for pointing out that Proposition~\ref{prop: infinite and full elasticity} cannot be derived as a direct consequence of a similar result already established for transfer Krull monoids (see Remark~\ref{remark: about transfer Krull monoids}).

\bigskip

\end{document}